\documentclass[12pt]{article}

\usepackage{amsmath,amsthm,amscd,amssymb,latexsym,upref,stmaryrd,epsfig,graphicx}
\usepackage[cp1251]{inputenc}
\usepackage[english]{babel}
\usepackage{mathtext}
\usepackage{amsfonts}
\usepackage {tikz}
\usepackage{hyperref}
\usepackage{cite}
\usepackage[numbers,sort&compress]{natbib}
\newcommand*{\mailto}[1]{\href{mailto:#1}{\nolinkurl{#1}}}
\usepackage{amsmath}
\usepackage{amsthm}

\newtheorem{theorem}{Theorem}[section]
\newtheorem{lemma}{Lemma}[section]
\newtheorem{remark}{Remark}[section]

\numberwithin{equation}{section}

\begin{document}
\thispagestyle{empty}
{\Large\bf   Some Bohr-type inequalities with two parameters for  }\\\\
{\centerline {\Large\bf bounded analytic functions } }\\\\
\begin{center}

{\bf Jianying Zhou  \quad  Qihan Wang  \quad Boyong Long$^*$}
\let\thefootnote\relax\footnotetext{\\This work is supported by NSFC (No.12271001) and Natural Science Foundation of Anhui
Province (2308085MA03), China. \\$^*$Corresponding author.\\ \quad Email: jianying820@163.com;\quad qihan@ahu.edu.cn; \quad boyonglong@163.com
}
\end{center}

\centerline {\small  (School of Mathematical Sciences, Anhui University, Hefei 230601, China)}
 \hspace*{\fill} \\\\
\noindent{\bf Abstract:}
{In this article, some Bohr inequalities for analytical functions on the unit disk are generalized  to the forms with two parameters. One of our results is sharp.
}

\medskip


\section{Introduction}

Let $|z|=r$, $\mathbb{D}:= \left\{z\in C:|z|<1\right\}$ and $\mathcal{BA}(\mathbb{D}):=\{f:f(z)=\sum_{k=0}^\infty a_k z^{k},\,\, z\in\mathbb{D}\,\, \mbox{and}\,\, |f(z)|<1\}$
be the  class of analytic function which mappings unit disk $\mathbb{D}$ into itself. The most classical Bohr's theorem \cite{bohr1914theorem} asserts that if $f\in\mathcal{BA}(\mathbb{D})$, then
   \begin{equation}\label{1.1}
   \sum_{k = 0}^\infty  |a_k|r^k  \leq 1
   \end{equation}
 for $r<r_{0}$. Inequality (\ref{1.1}) and $r_{0}$ are called Bohr inequality and Bohr radius, respectively.  Bohr takes  $r_{0}=1/6$ at first. Later, it is proved that the sharp result is  $r_{0}=1/3$. See [\cite{sidon1927satz},\cite{tomic1962theoreme}].

The classical Bohr's theorem has been generalized in many aspects. On one hand, the domain of the function has be changed into other different domains.
   Kaptano\u{g}lu\cite{kaptanoglu2005bohr}  has studied the Bohr's inequality for bounded analytic functions defined on elliptic regions  rather than unit disk.  Also, Evdoridis, Ponnusamy and  Rasila\cite{evdoridis2021improved}considered the case of  a shifted disk. More generally,  Kresin\cite{kresin2007hadamard} has researched the case of simply connected domain.

On the other hand, the bounded analytic function has been substituted into other classes of function.
Bhowmik\cite{bhowmik2021bohr} considered the Bohr radius problem of holomorphic functions in higher dimensions.
Allu\cite{allu2022bohr},Huang\cite{2021Bohr} and Kayumov\cite{2018Bohr} discussed Bohr inequalities for harmonic function.
Furthermore, Kayumov\cite{kayumov2022bohr,kayumov2020bohr,kayumov2021bohr} introduced the $Ces\acute{a}ro$ operator  and obtained some  Bohr type inequalities. Bhowmik\cite{bhowmik2021characterization} and Latacite\cite{lata2022bohr} discussed Bohr phenomenon for Banach space valued analytic functions
defined on the open unit disk $\mathbb{D}$.

Recently,  Wu\cite{wu2022some} and Hu\cite{hu2021bohr},\cite{hu2022bohr} explored some Bohr-type inequalities with one parameter or involving convex combination.

The following three theorems are related to the generalizations of classical Bohr's inequality.

\noindent{\bf Theorem 1.1} \cite{kayumov2017bohr}
Suppose that $f(z)\in \mathcal{BA}(\mathbb{D})$. Then
\begin{equation*}
|f(z)| + \sum_{k = 1}^\infty  |a_k||z|^k  \leq 1
\end{equation*}for
$|z|=r \leq \sqrt{5} - 2$.
The radius $r = \sqrt {5}  - 2$ is the best possible.

\noindent{\bf Theorem 1.2} \cite{Ming2018Bohr}
Suppose that $f(z)\in \mathcal{BA}(\mathbb{D})$, $N(\geq2)$ is an integer. Then
\begin{equation*}
|f(z)|+\sum_{k=N}^{\infty}|\frac{{f^{(k)}(z)}}{k!}||z|^{k}\leq 1
\end{equation*}for $|z|=r \leq r_N$,
where $r_N$ is the minimum positive toot of the equation $(1+r)(1-2r)(1-r)^{N-1}-2r^N=0$ and the radius  $r_N$ cannot be improved.

\noindent{\bf Theorem 1.3} \cite{wu2022some}
Suppose that $f(z)\in \mathcal{BA}(\mathbb{D})$. Let $S_r$ be the area of the image of the subdisk $\mathbb{D}_r={z:|z|<r}$  under the mapping $f$. Then for arbitrary $t \in (0,1]$, it holds that
\begin{equation*}
t\sum_{k = 0}^\infty |a_k||z|^k  + (1-t)\left(\frac{S_r}{\pi }\right) \leq 1
\end{equation*}for $|z|=r \leq r_*$,
where $S_r$ the radius $$r_*=r_*(t) = \left\{ \begin{array}{*{20}{l}}
r_t, \quad &t \in (0,\frac{9}{17})\\
\frac{1}{3}, \quad &t \in [\frac{9}{17},1]
\end{array}
 \right.$$ and the numbers $r = r_t$ is the unique positive root of the equation $$tr^3 + tr^2 + (4 - 5t)r - t = 0$$ in the interval $(0,\frac{1}{3})$.

In this article, we generalized the above Theorems 1.1-1.3  to two-parameter versions. The introduction  of two parameters rapidly complicates the corresponding problems, especially when we want verify the results if or not to be sharp, in other words,  to find the extremal functions. Even so, one of our results is proved to be sharp.

\section{ Some Lemmas }
In order to establish our main results, we need the following lemmas.

\begin{lemma}\label{lemma2.1}
\rm {{\cite{graham2003geometric}}Suppose $f(z) = \sum\limits_{k = 0}^\infty  a_kz^k$ is analytic in the unit disk $\mathbb{D}$ and $|f(z)|\leq 1$. Then
\begin{equation*}
|a_k| \leq 1 - |a_0|^2
\qquad
\text{for all}
\quad
k = 1,2,....
\end{equation*}}
\end{lemma}

\begin{lemma}\label{lemma2.2} $(\rm {Schwartz-Pick\: lemma})$
\rm{Suppose that $f(z) = \sum\limits_{k = 0}^\infty  a_kz^k$ is analytic function of $\mathbb{D}$ into itself. Then
\begin{equation*}
\frac{|f(z_1) - f(z_2)|}{|1 - \overline {f(z_1)} f(z_2)|} \leq \frac{|z_1 - z_2|}{|1 - \overline {z_1} z_2|}
\qquad
for \quad z_1, z_2 \in \mathbb{D},
\end{equation*}
and the equality holds for distinct ${z_1},{z_2} \in \mathbb{D}$ if and only if $f(z)$ is a M\"{o}bius transformation. Moreover,
\begin{equation*}
|f'(z)| \leq \frac{1 - |f(z)|^2}{1 - |z|^2}
\qquad
for \quad z \in \mathbb{D},
\end{equation*}
and the equality holds for some $z \in \mathbb{D}$ if and only if $f(z)$ is a M\"{o}bius transformation.}
\end{lemma}

\begin{lemma}\label{lemma2.3}{\rm \cite{dai2008note}}
\rm{Suppose that $f(z)\in \mathcal{BA}(\mathbb{D})$. Then
\begin{equation*}
|f^{(k)}(z)|\leq\frac{k!(1-|f(z)|^{2})}{(1-|z|^{2})^k}(1+|z|)^{k-1}
\qquad
\text{for all}
\quad
k = 1,2,....,
\end{equation*}and the equality holds for if and only if $f(z)$ is a M\"{o}bius transformation.}
\end{lemma}

\begin{lemma}\label{lemma2.4}\rm {Let $\alpha\in[\frac{4}{5},1] $ and $\beta\in(0,\alpha)$. Then it holds that
 the equalities
\begin{equation}\label{2.1}
2\alpha r^2+(2\beta+\alpha)r-\alpha=0
 \end{equation}
 and
\begin{equation}\label{2.2}
2r^2+(\beta-2\alpha-1)r+\alpha-\beta=0\end{equation}
have unique positive real roots $r^{*}_{1}$ and $r^{*}_{2}$  in the interval $(0,\frac{1}{2})$, respectively.
Furthermore,  $r^{*}_{1}>r^{*}_{2}$ .}
\end{lemma}

\begin{proof}
It is obvious that the equation (\ref{2.1}) has two real roots $R^{*}_{1}=\frac{-(2\beta+\alpha)+\sqrt{(2\beta+\alpha)^2+8\alpha^2}}{4\alpha}$  and $R^{*}_{2}=\frac{-(2\beta+\alpha)-\sqrt{(2\beta+\alpha)^2+8\alpha^2}}{4\alpha}$  with $ R^{*}_{2} < 0 <R^{*}_{1}$.  Furthermore, one observes that $R^{*}_{1}< \frac{1}{2}$. Thus, $r^{*}_{1}=R^{*}_{1}$.

In the same way, the equation (\ref{2.2}) has two real roots $R^{*}_{3}=\frac{(1+2\alpha-\beta)+\sqrt{(\beta-2\alpha-1)^2-8(\alpha-\beta)}}{4}$ and         $R^{*}_{4}=\frac{(1+2\alpha-\beta)-\sqrt{(\beta-2\alpha-1)^2-8(\alpha-\beta)}}{4}$ with $0 < R^{*}_{4} < R^{*}_{3}$. Observe that $R^{*}_{4}<\frac{1}{2}< R^{*}_{3}$ . So $r^{*}_{2}=R^{*}_{4}$.

Let\begin{align*}
h(\alpha,\beta)&=r^{*}_{1}-r^{*}_{2}\\
&=\frac{-(2\beta+\alpha)+\sqrt{(2\beta+\alpha)^2+8\alpha^2}}{4\alpha}-\frac{(1+2\alpha-\beta)-\sqrt{(\beta-2\alpha-1)^2-8(\alpha-\beta)}}{4}.
\end{align*} Then we need to show that $h(\alpha,\beta)> 0$ for $\alpha\in[\frac{4}{5},1] $ and $\beta\in(0,\alpha)$.
It follows that $$\sqrt{(2\beta+\alpha)^2+8\alpha^2}+\alpha\sqrt{(\beta-2\alpha-1)^2-8(\alpha-\beta)}> 2\alpha^2+2\alpha-\alpha\beta+2\beta,$$
$$\sqrt{[(2\beta+\alpha)^2+8\alpha^2][(\beta-2\alpha-1)^2-8(\alpha-\beta)]}>6\alpha^2+2\beta-\alpha\beta-3\alpha-2\beta^2,$$
and $$ (4-2\alpha)\beta^2+4\alpha^2 \beta+\alpha(-\alpha^2+\alpha+2)>0.$$

Let
\begin{align*}
g(\alpha,\beta)&=(4-2\alpha)\beta^2+4\alpha^2 \beta+\alpha(-\alpha^2+\alpha+2).
\end{align*}

Observe that $g(\alpha,\beta)$ is continuous and strictly increasing with respect to $\beta$ on the interval $(0,\alpha)$, $g(\alpha,0)=\alpha(-\alpha^2+\alpha+2)>0$ and $g(\alpha,\alpha)=\alpha(\alpha^2+5\alpha+2)>0$  imply that $g(\alpha,\beta)>0$. Therefore,  $g(\alpha,\beta)>0$ holds  for $\alpha\in[\frac{4}{5},1] \,\,\mbox{and} \,\, \beta\in(0, \alpha)$.

\end{proof}

\section{ Main results }

\begin{theorem}\label{theorem3.1}
\rm { Suppose that $f(z)\in \mathcal{BA}(\mathbb{D})$. Let $a:=|a_{0}|$. Then for arbitrary $  \alpha\in (0,1]$, $\beta\in (1-\alpha,\infty)$, it holds that
\begin{equation*}
\alpha|f(z)|+(1-\alpha)a+\beta \sum_{k=1}^{\infty}|a_{k}||z|^{k} \leq 1
\end{equation*} for $|z|=r\leq R_1$,
where  $$R_1=\left\{ \begin{array}{*{20}{l}}
\frac{1}{4\alpha-2}, \quad &\beta= \alpha-\frac{1}{2};\\
\quad &\quad\\
\frac{-(\alpha+\beta)+\sqrt {(\alpha+\beta)^2+(1+2\beta-2\alpha)}}{1+2\beta-2\alpha}
, \quad &\beta \neq \alpha-\frac{1}{2}.
\end{array}
 \right.$$}is the best possible.
\end{theorem}

\begin{proof}
According to the assumption,  Lemma\ref{lemma2.1} and \ref{2.2}, we obtain
\begin{align}\label{3.1}
|f(z)|\leq\frac{|z|+a}{1+a|z|}\quad for \quad z\in\mathbb{D},\\
\label{3.2}
|a_k| \leq 1 - |a_0|^2 \quad for \quad k = 1,2,....
\end{align}
Using the above two inequalities (\ref{3.1})and (\ref{3.2}), we have
\begin{align*}
&\alpha|f(z)|+(1-\alpha)a+\beta\sum_{k=1}^{\infty}|a_{k}||z|^{k}\\
\leq&\alpha(\frac{r+a}{1+ar})+(1-\alpha)a+\beta(1-a^2)\frac{r}{1-r}\\
=&\frac{\alpha(a+r)(1-r)+(1-\alpha)a(1+ar)(1-r)+\beta(1-a^2)(1+ar)r}{(1+ar)(1-r)}.
\end{align*}
Let $$A_1 (a,r,\alpha,\beta)=\frac{\alpha(a+r)(1-r)+(1-\alpha)a(1+ar)(1-r)+\beta(1-a^2)(1+ar)r}{(1+ar)(1-r)}.$$
Now, we need to show that $A_1 (a,r,\alpha,\beta)\leq1$ holds for $r\leq R_{1}$. It is equivalent to show $B_1 (a)\leq0$ for $r\leq R_{1}$, where
\begin{align}\label{3.3}
B_1 (a)=&\alpha(a+r)(1-r)+(1-\alpha)a(1+ar)(1-r)+\beta(1-a^2)(1+ar)r-(1+ar)(1-r)\notag\\
=&(1-a)g(a),
\end{align}
where $$g(a)=(\alpha r+\alpha ra-1-ar)(1-r)+\beta (1+a)r(1+ar).$$
Direct computations lead to
\begin{align}\label{3.4}
g'(a)&=(\alpha r-r)(1-r)+\beta r(1+ar)+\beta (1+a)r^2,\nonumber\\
g''(a)&=2\beta r^2>0.
\end{align}
The inequality (\ref{3.4}) implies that $g'(a)\geq g'(0) = r[(1-\alpha+\beta)r+(\alpha+\beta-1)]>(1-\alpha+\beta)r^2>0.$ Therefore, $g(a)$ increases  monotonically on the interval $[0,1)$ and
\begin{align}\label{3.5}
g(a)\leq g(1)=(1+2\beta-2\alpha)r^2+(2\alpha+2\beta)r-1.
\end{align}
 Observe that
 $\Delta :=(2\alpha+2\beta)^2+4(1+2\beta-2\alpha)>0$ for arbitrary $  \alpha\in (0,1]$ and $\beta\in (1-\alpha,\infty)$. Then
 the equation $(1+2\beta-2\alpha)r^2+(2\alpha+2\beta)r-1 = 0$ has two real roots

\begin{align*}{r_{1, \alpha, \beta}} =\frac{-(\alpha+\beta)+\sqrt {(\alpha+\beta)^2+(1+2\beta-2\alpha)}}{1+2\beta-2\alpha}\end{align*}and
\begin{align*}{r_{2, \alpha, \beta}} = \frac{-(\alpha+\beta)-\sqrt {(\alpha+\beta)^2+(1+2\beta-2\alpha)}}{1+2\beta-2\alpha}\end{align*}
with $1+2\beta-2\alpha \neq 0$.

Now we divide it into two cases to discuss.

{\bf Case 1.} If $\alpha\in(0,\frac{3}{4}]$, then   $1+2\beta-2\alpha>0$. It follows  that $r_{2, \alpha, \beta} < 0 < r_{1,\alpha,\beta}<1 $ and $g(1)\leq 0$ for $r \in [0,r_{1, \alpha, \beta}]$. Taking $R_1=r_{1,\alpha,\beta}$, then (\ref{3.3}) and (\ref{3.5}) lead to that $B_1 (a)\leq0 $ holds for $ r \leq R_{1}$.

{\bf Case 2.}  If $\alpha\in(\frac{3}{4},1]$, we divide it into three subcases to discuss.

{\bf Subcase 1.} $1+2\beta-2\alpha=0$ if $\beta = \alpha-\frac{1}{2}$. Then $g(1)=(4\alpha-2)r-1$ and $g(1)$ is monotonically increasing with respect to $r$. Taking $R_1=\frac{1}{4\alpha-2}$, then (\ref{3.3}) and (\ref{3.5}) lead to that $B_1 (a)\leq0 $ holds for $ r \leq R_{1}$.

{\bf Subcase 2.} $1+2\beta-2\alpha<0$ if $\beta < \alpha-\frac{1}{2}$. Then $ 0 < r_{1, \alpha, \beta} < r_{2,\alpha,\beta}<1 $. Taking $R_1= r_{1, \alpha, \beta}$, then (\ref{3.3}) and (\ref{3.5}) lead to that $B_1 (a)\leq0 $ holds for $ r \leq R_{1}$.

{\bf Subcase 3.} $1+2\beta-2\alpha > 0$ if $\beta > \alpha-\frac{1}{2}$. Taking $R_1= r_{1, \alpha, \beta}$, then we can obtain that $B_1 (a)\leq0 $ holds for $ r \leq R_{1}$ just like the case 1.

To show the sharpness of the number $R_1$, we consider the function
\begin{equation*}
f(z) = \frac{a - z}{1 - az} = a - (1 - a^2)\sum_{k = 1}^\infty  a^{k - 1}z^k ,\quad
z \in \mathbb{D}.
\end{equation*}
Taking $z=-r$, for this function, we have
\begin{align*}
&\alpha|f( - r)| + (1-\alpha)a+\beta\sum_{k = 1}^\infty  |a_k|r^k \\
= & \alpha\frac{a + r}{1 + ar} + (1 - \alpha)a + \beta(1 - a^2)\frac{r}{1 - ar} \\
= & \frac{\alpha(a + r)(1 - ar) + (1 - \alpha)a(1 + ar)(1 - ar) + \beta(1 - a^2)(1+ar)r}{(1 - a^2r^2)}.
\end{align*}
Let $$H(a,r,\alpha,\beta)=\frac{\alpha(a + r)(1 - ar) + (1 - \alpha)a(1 + ar)(1 - ar) + \beta(1 - a^2)(1+ar)r}{(1 - a^2r^2)}.$$
Now we just need to show that if $r>R_1$, then there  exists an $a$ such that $H(a,r,\alpha,\beta)>1$. That is equivalent to show $h (a)>0$ for $r>R_1$, where
\begin{align*}
h(a) &=\alpha(a + r)(1 - ar) + (1 - \alpha)a(1 + ar)(1 - ar) + \beta(1 - a^2)(1+ar)r-(1-ar)(1+ar) \\
&=(1 - a)[(1-\alpha +\beta)r^2a^2+(\alpha r+\beta r-\alpha r^2+\beta r^2)a+(\alpha r+\beta r-1)].
\end{align*}
Let $$\rho (a)=(1-\alpha +\beta)r^2a^2+(\alpha r+\beta r-\alpha r^2+\beta r^2)a+(\alpha r+\beta r-1).$$ It follows that  $\rho '(a)=2(1-\alpha +\beta)r^2 a+\alpha r+\beta r-\alpha r^2+\beta r^2 \geq 0$. This means that $\rho (a) < \rho (1) =(1+2\beta-2\alpha)r^2+(2\alpha+2\beta)r-1 $ for each fixed $\alpha \in (0,1]$, $\beta\in (1-\alpha,\infty)$ and $r \in [0,1)$. Observe that  $\rho (1)=g(1)$. So, if $r > R_1$, then $\rho (1)>0$. Thus, by the continuity of $\rho (a)$, we have $$\lim_{a\rightarrow 1^-} \rho (a) = \rho (1) > 0.$$
Therefore, if $r > R_1$ , then there exists an $a\in [0,1)$ such that $\rho (a)>0$, namely $h(a)>0$. This proves the sharpness of $r = R_1$.
\end{proof}

\begin{remark}
\rm{(1) If $\alpha = 1$ and $\beta=1$, then Theorem 3.1 reduces to Theorem 1.1.\\
(2) If $\alpha = 1$ and $\beta\in(0,\infty)$, then Theorem 3.1 reduces to Corollary 3.1 of \cite{wu2022some}.}
\end{remark}

\begin{theorem}\label{theorem3.2}
\rm {Suppose that $f(z)\in \mathcal{BA}(\mathbb{D})$. Then for arbitrary $\alpha \in (0,1]$ and $\beta\in (0,\infty)$, it holds that
\begin{equation*}
\alpha|f(z)|+\beta \sum_{k=1}^{\infty}|a_{k}||z|^{k} \leq 1
\end{equation*} for $|z|=r\leq R_2$,
where  $$R_2 = \left\{ \begin{array}{*{20}{l}}
\frac{\alpha}{\alpha+2}, \quad &\beta= \frac{1}{2};\\
\quad &\quad\\
\frac{(\alpha+2\beta+1)-\sqrt{(\alpha+2\beta+1)^2-4\alpha(1-2\beta)}}{2(1-2\beta)}
, \quad &\beta \neq \frac{1}{2}.
\end{array}
 \right.$$}
\end{theorem}

\begin{proof}
Let $|a_0|=a$. Inequalities (\ref{3.1}) and (\ref{3.2}) lead to that
\begin{align*}
&\alpha|f(z)|+\beta\sum_{k=1}^{\infty}|a_{k}||z|^{k}\\
\leq&\alpha(\frac{r+a}{1+ar})+\beta(1-a^2)\frac{r}{1-r}\\
=&\frac{\alpha(a+r)(1-r)+\beta(1-a^2)(1+ar)r}{(1+ar)(1-r)}.
\end{align*}
Let $$A_2 (a,r,\alpha,\beta)=\frac{\alpha(a+r)(1-r)+\beta(1-a^2)(1+ar)r}{(1+ar)(1-r)} .$$
Now, we need to show that $A_2 (a,r,\alpha,\beta)\leq1$ holds for $r\leq R_2$. It is equivalent to show $B_2 (a)\leq0$ holds for $r\leq R_2$, where
\begin{align*}
B_2 (a)=\alpha(a+r)(1-r)+\beta(1-a^2)r(1+ar)-(1+ar)(1-r).
\end{align*}
It follows that
\begin{align}\label{3.6}
B_2 '(a)=&\alpha(1-r)-2\beta ar(1+ar)+\beta r^2(1-a^2)-r(1-r),\notag\\
B_2 ''(a)=&-2\beta r(1+ar)-4\beta r^{2}a=-6\beta r^{2}a-2\beta r<0.
\end{align}
The  inequality(\ref{3.6}) implies that $$B_2 '(a)> B_2 '(1) = (1-2\beta)r^2 - (\alpha+2\beta+1)r + \alpha.$$\\
Let $$u(r)=(1-2\beta)r^2 - (\alpha+2\beta+1)r + \alpha.$$ If $\alpha \in (0,1]$ and $\beta \neq \frac{1}{2}$, observe that $\Delta:=(\alpha+2\beta+1)^2-4\alpha(1-2\beta)>0$. Then the equation $u(r)=0$ has two real roots
\begin{align*}r_{3,\alpha,\beta}=\frac{(\alpha+2\beta+1)+\sqrt{(\alpha+2\beta+1)^2-4\alpha(1-2\beta)}}{2(1-2\beta)}\end{align*}and
\begin{align*}r_{4,\alpha,\beta}=\frac{(\alpha+2\beta+1)-\sqrt{(\alpha+2\beta+1)^2-4\alpha(1-2\beta)}}{2(1-2\beta)}.\end{align*}
Next, we divide it into three cases to discuss.

	{\bf Case 1.} If $\beta\in(0,\frac{1}{2})$, then $1-2\beta>0$ . By simple calculations we obtain that $0< r_{4,\alpha,\beta}< 1<r_{3,\alpha,\beta}$. It follows that $u(r)\geq 0$ holds for $r \leq r_{4,\alpha,\beta}$. Then we have $B_2 '(a)> B_2 '(1)\geq0$. Taking $R_2=r_{4,\alpha,\beta}$, we have $ B_2 (a)< B_2 (1)=(\alpha-1)(1+r)(1-r)\leq 0$ for $r\in[0,R_2]$.

	{\bf Case 2.}  If $\beta\in(\frac{1}{2},\infty)$, then $1-2\beta<0$. So  $r_{3,\alpha,\beta} < 0 < r_{4,\alpha,\beta} < 1$ and $B_2 '(a)> B_2 '(1)=u(r)\geq0$ holds for $r \leq r_{4,\alpha,\beta}$. Taking $R_2=r_{4,\alpha,\beta}$, we have $ B_2 (a)< B_2 (1)=(\alpha-1)(1+r)(1-r)\leq 0$ for $r\in[0,R_2]$.

{\bf Case 3.}  If $\beta=\frac{1}{2}$, then $u(r)=-(\alpha+2)r+\alpha$. Observe that $u(r)$ is a monotonically decreasing function of $r$. Taking $R_2=\frac{\alpha}{\alpha+2}$. Then if $r\in[0,R_2]$, we  can also obtain that $B_2 '(a)> B_2 '(1)=u(r)\geq 0$  and $ B_2 (a)< B_2 (1)=(\alpha-1)(1+r)(1-r)\leq 0$ .

\end{proof}

\begin{theorem}\label{theorem3.3}
\rm{Suppose that $f(z)\in \mathcal{BA}(\mathbb{D})$. Then for arbitrary  $\alpha\in (0,1]$ and $\beta\in [\alpha,\infty)$, it holds that
\begin{equation*}
\alpha|f(z)|+\beta\sum_{k=1}^{\infty}|\frac{{f^{(k)}(z)}}{k!}||z|^{k}\leq 1
\end{equation*}
for $|z|=r\leq R_{3}<\frac{1}{2}$,
where  $R_3 =\min(r^*_1,r^{*}_{3})$. The radii  $r^{*}_{1}$ and $r^{*}_{3}$ are the unique positive real roots of the equations $$2\alpha r^2+(2\beta+\alpha)r-\alpha=0$$ and $$(4-2\alpha)r^3+(2+2\beta -3\alpha)r^2-(2\beta+2)r+\alpha=0$$ in the interval $(0,\frac{1}{2})$, respectively.}
\end{theorem}

\begin{proof}
Let $ a:=|a_{0}|$. According to the assumption and Lemma \ref{lemma2.3}, we have
\begin{align*}
&\alpha|f(z)|+\beta\sum_{k=1}^{\infty}|\frac{{f^{(k)}(z)}}{k!}||z|^{k}\\
\leq&\alpha|f(z)|+\beta(1-|f(z)|^{2})\sum_{k=1}^{\infty}\frac{(1+r)^{k-1}}{(1-r^{2})^k}r^k\\
=&\alpha|f(z)|+\beta\frac{(1-|f(z)|^{2})r}{(1+r)(1-2r)}\\
=&\alpha|f(z)|-\frac{\beta r}{(1+r)(1-2r)}|f(z)|^{2}+\frac{\beta r}{(1+r)(1-2r)}
\end{align*}
Let $$Q(|f(z)|)=\alpha |f(z)|-\frac{\beta r}{(1+r)(1-2r)}|f(z)|^{2},\qquad (|f(z)|\in[0,1)).$$ It follows that $Q'(|f(z)|)=\alpha-\frac{2\beta r}{(1+r)(1-2r)}|f(z)|$. By simple calculations we can obtain that $2\alpha r^2+(2\beta+\alpha)r-\alpha\leq0$ holds for $r\leq r^{*}_{1}<\frac{1}{2}$, then $\frac{2\beta r}{(1+r)(1-2r)}\leq\alpha$. Thus, $Q'(|f(z)|)=\alpha-\frac{2\beta r}{(1+r)(1-2r)}|f(z)|>\alpha-\frac{2\beta r}{(1+r)(1-2r)}\geq0$. Therefore, $Q(|f(z)|)$ is a monotonically increasing function of $|f(z)|$  for each fixed $r\in[0, r^{*}_{1}]$.
Using the inequality(\ref{3.1}),
\begin{align*}
&\alpha|f(z)|-\frac{\beta r}{(1+r)(1-2r)}|f(z)|^{2}+\frac{\beta r}{(1+r)(1-2r)}\\
\leq&\alpha(\frac{a+r}{1+ar})+\beta[1-(\frac{a+r}{1+ar})^{2}]\frac{r}{(1+r)(1-2r)}\\
=&\frac{\alpha(a+r)(1+ar)(1-2r)+\beta(1-a^2)(1-r)r}{(1+ar)^2 (1-2r)}.
\end{align*}
Let $$A_3 (a,r,\alpha,\beta)=\frac{\alpha(a+r)(1+ar)(1-2r)+\beta(1-a^2)(1-r)r}{(1+ar)^2 (1-2r)}.$$
Now, we need to show that $A_3 (a,r,\alpha,\beta)\leq1$ holds for $r\leq R_{3}$. It is equivalent to show $B_3 (a)\leq0$ holds for $r\leq R_{3}$, where
\begin{align*}
B_3 (a)=\alpha(a+r)(1+ar)(1-2r)+\beta(1-a^2)(1-r)r-(1+ar)^2 (1-2r).
\end{align*}
Direct computations lead to
\begin{align}\label{3.7}
B_3 '(a)=&\alpha (1 +ar)(1-2r)+\alpha r(a + r)(1 - 2r)-2\beta a(1 - r)r-2r(1 +ar)(1-2r),\notag\\
B_3 ''(a)
=&2rs(r),
\end{align}
where $s(r) =2r^2+(\beta-2\alpha-1)r+\alpha-\beta$.
Observe that $s(0) =\alpha-\beta\leq0$ and $s(\frac{1}{2})=-\frac{\beta}{2}<0$. So,  $s(r)\leq0$ holds for $ r \leq r^{*}_{1}<\frac{1}{2}$. Then (\ref{3.7}) leads to $B_3 ''(a) \leq 0$. It follow  that $B_3 '(a)$ is an decreasing function with respect to $a$. Let $$p(r) =B_3 '(1)=(4-2\alpha)r^3+(2+2\beta -3\alpha)r^2-(2\beta+2)r+\alpha.$$ Observe that $p(0)=\alpha>0$, $p(\frac{1}{2})=-\frac{\beta}{2}<0$ and $p''(r)=12(2-\alpha)r+2(2+2\beta-3\alpha)\geq 0$. This implies that  there is a unique $r^{*}_{3} \in (0,\frac{1}{2})$ such that $p(r^{*}_{3}) = 0$ and $p(r)\geq 0$ holds for $r\leq r^{*}_{3}$. So if $r \in [0,\min(r^{*}_{1},r^{*}_{3})]$, we have $B_3 '(a) > B_3 '(1)=p(r)\geq0$. Therefore,  $B_3 (a)< B_3 (1)=(\alpha-1)(1+r)^2 (1-2r)\leq 0$ holds for $r \in [0,R_3]$.
\end{proof}

\begin{remark}
\rm{(1) If $\alpha = 1$ and $\beta=1$, then Theorem 3.3 reduces to Theorem 1.2.\\
(2) Actually,
\begin{align*}&r^{*}_{1}=\frac{-(2\beta+\alpha)+\sqrt{(2\beta+\alpha)^2+8\alpha^2}}{4\alpha},\\
&r^{*}_{3}=
 \sqrt[3]{-\frac{A}{2}+\frac{1}{27}\sqrt{\frac{A^2+4B^3}{4(4-2\alpha)^3}}}+
 \sqrt[3]{-\frac{A}{2}-\frac{1}{27}\sqrt{\frac{A^2+4B^3}{4(4-2\alpha)^3}}}-\frac{2+2\beta-3\alpha}{3(4-2\alpha)}\end{align*}
with $A=54\alpha^3+16\beta^3+216\alpha^2\beta-144\alpha\beta^2-236\alpha^2+192\beta^2-288\alpha\beta+336\beta+304$ and
$B=-9\alpha^2-4\beta^2+24\alpha\beta+24\alpha-32\beta-4$.}
\end{remark}

Let $\alpha$ and $\beta$ be in different intervals,  we have similar results as follows.

\begin{theorem}\label{theorem3.4}
\rm{Suppose that $f(z)\in \mathcal{BA}(\mathbb{D})$. Then for arbitrary  $\alpha\in [\frac{4}{5},1]$ and $\beta\in (0,\alpha)$, it holds that
\begin{equation*}
\alpha|f(z)|+\beta\sum_{k=1}^{\infty}|\frac{{f^{(k)}(z)}}{k!}||z|^{k}\leq 1
\end{equation*}
for $|z|=r\leq R_{4}<\frac{1}{2}$,
where $ R_{4}= r^{*}_{2}=\frac{(1+2\alpha-\beta)-\sqrt{(\beta-2\alpha-1)^2-8(\alpha-\beta)}}{4}$.}

\end{theorem}

\begin{proof}

The first part of the proof is same as  that of  Theorem \ref{theorem3.3}.  After \ref{3.7},observe that $s(r)$ is decreasing for $r<(1+2\alpha-\beta)/4$ and  $(1+2\alpha-\beta)/4>r^*_2$.  Lemma \ref{lemma2.4} says that  $s(r^{*}_{2}) = 0$ and $r^{*}_{1}> r^{*}_{2}$ for $\alpha\in [\frac{4}{5},1]$ and $\beta\in (0,\alpha)$. Thus we have
  $s(r)\geq0$ for $r\in[0,r^{*}_{2}]$. Then (\ref{3.7}) implies that $ B_3 ''(a) \geq 0$. It follows  that $B_3 '(a)$ is an increasing function with respect to $a$. Let $$B_3'(0)=(1-2r)(\alpha r^2-2r+\alpha)=(1-2r)q(r). $$
   Observe that $q(r)\geq0$ for $r\in[0,\frac{2-\sqrt{4-4\alpha^2}}{2\alpha}]$. Because  $r^{*}_{2}<\frac{1}{2}$ and   it holds that $\frac{2-\sqrt{4-4\alpha^2}}{2\alpha}\geq\frac{1}{2}$ for $\alpha\in [\frac{4}{5},1]$, we have $q(r)>0$  for $r\leq r^{*}_{2}$. Thus $B_3 '(0)>0$. It follows that $B_3 '(a)>0$. Therefore, $B_3 (a)< B_3 (1)=(\alpha-1)(1+r)^2 (1-2r)\leq 0$ for $r \leq r^{*}_{2}$.

\end{proof}

\begin{theorem}\label{theorem3.5}
\rm { Suppose that $f(z)\in \mathcal{BA}(\mathbb{D})$. Let $S_r$ be the area of the image of the subdisk $\mathbb{D}_r=\left\{z\in C:|z|< r\right\}$ under the mapping $f(z)$. Then for arbitrary $ \alpha \in (0,1] $ and $ \beta \in (0,\infty)$,  it holds that
\begin{equation*}
\alpha\sum_{k = 0}^\infty  |a_k||z|^k +\beta\left(\frac{S_r}{\pi }\right) \leq 1
\end{equation*}
for  $|z|=r \leq {R_5}$, where  $$R_5 = \left\{ \begin{array}{*{20}{l}}
\frac{1}{3}, \quad &\beta \in (0,\frac{8\alpha}{9});\\
\quad &\quad\\
\widetilde{R_{5}}
, \quad &\beta \in [\frac{8\alpha}{9},\infty)
\end{array}
 \right.$$ and $\widetilde{R_{5}}$ is the unique positive real root of the equation $$\alpha r^3+\alpha r^2+(4\beta-\alpha)r-\alpha=0$$ in the interval $(0,\frac{1}{3}]$.}
\end{theorem}

\begin{proof}
The definition of ${S_r}$  implies  that
\begin{align}\label{3.8}
\frac{S_r}{\pi } = \frac{\iint_{|z| < r} |f'(z)|^2dxdy}{\pi } = \sum_{k = 1}^\infty  k|a_k|^2r^{2k}  \leq (1 - |a_0|^2)^2\sum_{k = 1}^\infty  kr^{2k}  = \frac{(1 - |a_0|^2)^2r^2}{(1 - r^2)^2}.
\end{align}
Let $a:=|a_{0}|$. Using the above inequalities (\ref{3.2}) and (\ref{3.8}) , we have
\begin{align*}
&\alpha\sum_{k = 0}^\infty  |a_k|r^k +\beta\left(\frac{S_r}{\pi}\right)\notag\\
\leq & \alpha a + \alpha(1 - a^2)\frac{r}{1 - r} + \beta(1 - a^2)^2\frac{r^2}{(1 - r^2)^2}\notag\\
=& \frac{\alpha a(1 - r^2)^2 + \alpha r(1 - a^2)(1 + r)(1 - r^2) +\beta r^2(1 - a^2)^2}{(1 - r^2)^2}.
\end{align*}
Let $$A_5 (a,r,\alpha,\beta)=\frac{\alpha a(1 - r^2)^2 + \alpha r(1 - a^2)(1 + r)(1 - r^2) +\beta r^2(1 - a^2)^2}{(1 - r^2)^2}.$$
It is sufficient for us to prove that $A_5 (a,r,\alpha,\beta)\leq1$ holds for $r \leq R_5$. Actually, we just need to prove $B_5 (a) \leq 0$ for $r \leq R_5$, where
\begin{align*}
 B_5 (a) &= \alpha a(1 - r^2)^2 + \alpha r(1 - a^2)(1 + r)(1 - r^2) +\beta r^2(1 - a^2)^2-(1 - r^2)^2.
\end{align*}
It follows that
\begin{align}\label{3.9}
&B_5 '(a)=\alpha (1 - r^2)^2-2a\alpha r(1 + r)(1 - r^2)-4\beta r^2a(1 - a^2),\notag\\
&B_5 ''(a)=-2\alpha r(1 + r)(1 - r^2)-4\beta r^2+12\beta a^2r^2,\notag\\
&B_5 '''(a)=24\beta ar^2>0.
\end{align}
The inequality (\ref{3.9}) implies that
$B_5 ''(a) < B_5 '' (1)= 2r[\alpha r^3+\alpha r^2+(4\beta-\alpha)r-\alpha]$ for  $a\in [0,1)$.
Let $$\varphi(r) = \alpha r^3+\alpha r^2+(4\beta-\alpha)r-\alpha.$$
Then we have
\begin{align}\label{3.10}
\varphi'(r) &=3\alpha r^2+2\alpha r + 4\beta - \alpha,\notag\\
\varphi''(r)& = 6\alpha r+2\alpha>0.
\end{align}
By the inequality (\ref{3.10}), $\varphi '(r)$ is an increasing function of $r\in[0,1)$. Observe that $\varphi'(0)=4\beta-\alpha$ and $\varphi'(\frac{1}{3})=4\beta>0$.

 Now we divide it into three cases to discuss.

{\bf Case 1.} If $\beta\in(0,\frac{\alpha}{4}]$, then $\varphi'(0)=4\beta-\alpha\leq0$. Since $\varphi'(0)\varphi'(\frac{1}{3})\leq0$, there exists a unique real root $r^*_{\alpha,\beta}\in(0,\frac{1}{3}]$ , such that $\varphi'(r^*_{\alpha,\beta})=0$.  So $\varphi(r)$ is decreasing in $[0,r^*_{\alpha,\beta})$ and increasing in $ (r^*_{\alpha,\beta},\frac{1}{3}]$. Observe that $$\varphi (0)=-\alpha<0 \,\,\, \mbox{and}\,\,\,  \varphi (\frac{1}{3})=\frac{4(9\beta-8\alpha)}{27}<0.$$ It implies that $\varphi(r)\leq0$ holds for $r\leq\frac{1}{3}$. It follows that $ B_5 ''(a) < B_5 ''(1)=2r\varphi (r) \leq 0$ and $B_5 '(a) > B_5 '(1)=(1+r)(1-r^2)(1-3r)>0$. Thus $B_5 (a) \leq B_5 (1) = 0$ holds for $r \in [0,\frac{1}{3}]$.

{\bf Case 2.}  If $\beta\in(\frac{\alpha}{4},\frac{8\alpha}{9})$, then $\varphi '(0) =-\alpha+4\beta>0$. Inequality (\ref{3.10}) implies that $\varphi'(r)\geq \varphi'(0)>0$. So, $\varphi(r)$ is an increasing function for $r\in [0,\frac{1}{3}]$. Because $\varphi (\frac{1}{3})=\frac{4(9\beta-8\alpha)}{27}<0$,  $\varphi(r)\leq0$ holds for $r\leq\frac{1}{3}$. Similar to that of Case 1, we also have  $B_5 (a) \leq B_5 (1) = 0$ holds for $r \in [0,\frac{1}{3}]$.

 {\bf Case 3.}  If $ \beta \in[\frac{8\alpha}{9},\infty)$, then $\varphi '(0) =-\alpha+4\beta>0$. By the inequality(\ref{3.10}), we have $\varphi'(r)\geq \varphi'(0)>0 $. So that $\varphi(r)$ is an increasing function for $r\in [0,\frac{1}{3}]$. Because $$\varphi (0)=-\alpha<0 \,\, \mbox{and}\,\,\varphi (\frac{1}{3})=\frac{4(9\beta-8\alpha)}{27}\geq0,$$   there is a unique $ \widetilde{R_{5}}\in(0,\frac{1}{3}]$ such that $\varphi(\widetilde{R_{5}})=0$ and $\varphi(r)\leq0$ holds for $r\leq \widetilde{R_{5}}$. Similarly, if $r \in [0,\widetilde{R_{5}}]$, $B_4 (a) \leq B_5 (1) = 0$ holds.\\
\end{proof}

\begin{remark}
\rm {(1) If $\alpha\in(0,1]$ and $\beta=1-\alpha$, then Theorem 3.5 reduces to Theorem 1.3.

(2)  Actually, $\widetilde{R_{5}}=-\frac{1}{3}+\sqrt[3]{\frac{8\alpha+18\beta}{27\alpha}
 +\frac{2}{3}\sqrt{\frac{\beta(8\alpha^2-13\alpha\beta+16\beta^2)}{3\alpha}}}
 +\sqrt[3]{\frac{8\alpha+18\beta}{27\alpha}
 -\frac{2}{3}\sqrt{\frac{\beta(8\alpha^2-13\alpha\beta+16\beta^2)}{3\alpha}}}$.  }
\end{remark}

\begin{theorem}\label{theorem3.6}
\rm{Suppose that $f(z)\in \mathcal{BA}(\mathbb{D})$. Let $S_r$ be the area of the image of the subdisk $\mathbb{D}_r=\left\{z\in C:|z|< r\right\}$ under the mapping $f$. Then for arbitrary $ \alpha \in (0,1]$ and $ \beta \in (0, \infty) $, it holds that
\begin{equation*}
\alpha\sum_{k = 0}^\infty  |a_k||z|^k  + (1 - \alpha)a+\beta(\frac{S_r}{\pi }) \leq 1
\end{equation*}
for $|z|= r \leq {R_6}$,where  $$R_6 = \left\{ \begin{array}{*{20}{l}}

\frac{1}{2\alpha+1}, \quad &\beta \in (0,\frac{2\alpha^2(\alpha+1)^2}{(2\alpha+1)^2});\\
\quad &\quad\\
\widetilde{R_6}, \quad &\beta \in [\frac{2\alpha^2(\alpha+1)^2}{(2\alpha+1)^2},\infty)
\end{array}
 \right.$$and the radius $\widetilde{R_6}$ is the unique positive root of the equation $$\alpha r^3+\alpha r^2+(4\beta-\alpha)r-\alpha=0$$ in the interval $(0,\frac{1}{2\alpha+1}]$.}
\end{theorem}

The proof is similar to Theorem 3.5. We omit it.\\

\medskip

{\bf Declarations}

{\bf Availability of data}  The data that support the findings of this study are included in this paper.

{\bf Conflict of interest}   The authors have not disclosed any competing interests.

\end{document}